\documentclass[11pt]{amsart}
\usepackage{amsmath,amssymb,amsthm,tikz,hyperref}
\usetikzlibrary{arrows,automata,positioning}

\newtheorem{prop}{Proposition}[section]
\newtheorem{thm}[prop]{Theorem}
\newtheorem{lmm}[prop]{Lemma}

\theoremstyle{remark} \theoremstyle{definition}

\newtheorem{remark}[prop]{Remark}
\newtheorem{definition}[prop]{Definition}

\newcommand{\R}{\mathbb{R}}
\newcommand{\Z}{\mathbb{Z}}
\newcommand{\N}{\mathbb{N}}
\newcommand{\K}{\mathcal{K}}
\newcommand{\D}{\mathcal{D}}
\numberwithin{equation}{section}

\title[Intersecting the Twin dragon with rational lines]{Intersecting the Twin dragon with rational lines}

\author[S. Akiyama, P. Gro{\ss}kopf, B. Loridant and W. Steiner]{Shigeki Akiyama, Paul Gro{\ss}kopf, Beno\^it Loridant and Wolfgang Steiner}

\address{Tsukuba University, Institute of Mathematics\\Tennodai-1-1-1\\ Tsukuba 350-8571, Japan}
\email{akiyama@math.tsukuba.ac.jp }
\address{Universit\'e libre de Bruxelles\\Boulevard du Triomphe\\ 1050 Bruxelles, Belgium}
\email{paul.grosskopf@gmx.at}
\address{Leoben University\\Franz Josefstrasse 18\\8700 Leoben, Austria}
\email{benoit.loridant@unileoben.ac.at}
\address{Universit\'e de Paris, CNRS, IRIF, F--75006 Paris, France}
\email{steiner@irif.fr}

\date{\today}

\dedicatory{Dedicated to Professor J\"org Thuswaldner on the occasion of his $50^{th}$ birthday}

 \subjclass[2000]{}

\begin{document}
\begin{abstract}
The Knuth Twin Dragon is a compact subset of the plane with fractal boundary of Hausdorff dimension $s = (\log \lambda)/(\log \sqrt{2})$, $\lambda^3 = \lambda^2 + 2$. 
Although the intersection with a generic line has Hausdorff dimension~$s-1$, we prove that this does not occur for lines with rational parameters. We further describe the intersection of the Twin Dragon with the two diagonals as well as with various axis parallel lines. 
\end{abstract}
\maketitle

\begin{section}{Introduction}
We investigate the intersections of the Knuth Twin Dragon with rational lines. 
Let $\alpha=-1+i$, then 
\[
\K=\left\{\sum_{k=1}^\infty \frac{d_k}{\alpha^k} \,:\, d_k \in \{0,1\}\right\}
\]
is the Knuth Twin Dragon. The Hausdorff dimension of its boundary $\partial K$ is $\mathfrak{s} = \frac{\log\lambda}{\log\sqrt{2}} \approx 1.5236$, where $\lambda$ is the real number satisfying $\lambda^3 = \lambda^2+2$. 
For lines 
\begin{equation} \label{e:Delta}
\Delta_{p,q,r} = \{x+iy \in \mathbb{C} \,:\, px+qy=r\}
\end{equation}
with $p,q,r \in \Z$, we show that the $\alpha$-expansions of $\K \cap \Delta_{p,q,r}$ are recognized by a finite automaton.

By a result of John Marstrand~\cite{M54}, the intersection of $\partial \K$ with Lebesgue almost all lines going through $\K$ has Hausdorff dimension $\mathfrak{s}-1$, meaning that in the set of all parameter triples $(p,q,r)\in \R^3$ for which $\Delta_{p,q,r}\cap \K \neq \emptyset$, the exceptional cases form a Lebesgue null set.  We obtain here that the Hausdorff dimension of the intersection of the boundary of the Twin Dragon with rational lines is never equal to $\mathfrak{s}-1$.

Further we revisit results by Shigeki Akiyama and Klaus Scheicher~\cite{AS05} and add uncountably many examples of horizontal, vertical, and diagonal lines. 

We mention that similar results were obtained in~\cite{MS13} for lines intersecting the Sierpinski carpet $F$. The set $F$ has Hausdorff dimension $\frac{\log8}{\log3}$. Manning and Simon showed that,  given a slope $\alpha\in\mathbb{Q}$, the intersection of $F$ with the line $y=\alpha x+\beta$ is strictly less than $\frac{\log8}{\log3}-1$ for Lebesgue almost every $\beta$. 

\begin{section}{Main statement and proof}
We first recall the notions of a canonical number system and its fundamental domain. 
Let $\beta$ be an algebraic integer and $\mathcal{N}=\{0,1,\dots,|N(\beta)|-1\}$, where $N(x)$ denotes the norm of $x$ over $\mathbb{Q}(\beta)/\mathbb{Q}$.  The pair $(\beta, \mathcal{N})$ is called a \emph{canonical number system} (CNS) if each $\gamma \in \Z[\beta]$ admits a representation of the form
\begin{equation} \label{e:gamma}
\gamma = \sum_{k=0}^n d_k \beta ^k, \quad d_k \in \mathcal{N}. 
\end{equation}
We call $\beta$ the  \emph{radix} or \emph{base} and $\mathcal{N}$ the set of \emph{digits}.
The representation \eqref{e:gamma} is unique up to leading zeros. 

The Knuth Twin Dragon $\K$ appears as the \emph{fundamental domain} of the CNS $(\alpha,\mathcal{N})$, where $\alpha=-1+ i$ is the root of the polynomial $x^2-2x-2$ and $\mathcal{N}=\{0,1\}$. The fundamental domain of a CNS is the set of all numbers that can be expressed with purely negative exponents. 
Since $\alpha^4=-4$, it is often useful to consider groups of four digits:
\[
\sum_{k=1}^\infty \frac{d_k}{\alpha^k} = \sum_{k=1}^\infty \frac{\sum_{j=0}^3 d_{4k-j} \alpha^j}{\alpha^{4k}} = \sum_{k=1}^\infty \frac{b_k}{(-4)^k},
\]
with the possibilities for $b_k = \sum_{j=0}^3 d_{4k-j} \alpha^j$ being 
\[
\begin{array}{llll}
[0000]_{\alpha} = 0, & [0001]_{\alpha} = 1, & [0010]_{\alpha}= -1{+}i, & [0011]_{\alpha} = i, \\ {}
[0100]_{\alpha} = -2i, & [0101]_{\alpha} = 1{-}2i, & [0110]_{\alpha} = -1{-}i, & [0111]_{\alpha} = -i, \\ {}
[1000]_{\alpha} = 2{+}2i, & [1001]_{\alpha} = 3{+}2i, & [1010]_{\alpha} = 1{+}3i, & [1011]_{\alpha} = 2{+}3i, \\ {}
[1100]_{\alpha} = 2, & [1101]_{\alpha} = 3, & [1110]_{\alpha} = 1{+}i, & [1111]_{\alpha} = 2{+}i.
\end{array}
\]
In other words, we have
\[
\K = \left\{\sum_{k=1}^\infty \frac{b_k}{(-4)^k} \,:\, b_k \in \mathcal{D}\right\},
\]
with
\[
\mathcal{D} = \{-1{-}i, -1{+}i, -2i, -i, 0, i, 1{-}2i, 1, 1{+}i, 1{+}3i, 2, 2{+}i, 2{+}2i, 2{+}3i, 3, 3{+}2i\} 
\]
Points in the intersection of $\K$ with lines $\Delta_{p,q,r} = \{x+iy \,:\, px+qy=r\}$ can now be characterized by their digit expansion in the following way.

\begin{lmm} \label{l:1}
We have $z \in \K \cap \Delta_{p,q,r}$ if and only if there is a digit sequence $b_1b_2\cdots \in \D^{\N}$ with 
\[
z= \sum_{k=1}^\infty  \frac{b_k}{(-4)^k} \quad \mbox{and} \quad  r = \sum_{k=1}^\infty \frac{p\, \mathfrak{R}(b_k)+q\, \mathfrak{I}(b_k)}{(-4)^k}.
\]
\end{lmm}

Here, $\mathfrak{R}(b)$ denotes the real part and $\mathfrak{I}(b)$ denotes the imaginary part of $b \in \mathbb{C}$. 

We will show that we can characterize the digit expansion of the points in the intersection $\Delta_{p,q,r} \cap \K$ via a B\"uchi automaton, that is a finite automaton that accepts infinite paths. Using this representation we will be able to calculate the Hausdorff dimension of the intersection $\K \cap \Delta_{p,q,r}$ as well as the Hausdorff dimension of $\partial K \cap \Delta_{p,q,r}$. 

\begin{definition}
A \emph{B\"uchi automaton} is a $5$-tuple $(Q,A,E,I,T)$, where $Q=\{q_1,\dots,q_N\}$ is a finite set of \emph{states}, $A$ is a finite \emph{alphabet}, $E\subset Q\times A \times Q$ is a set of \emph{edges} and $I,T\subset Q$ the set of \emph{initial} and \emph{terminal states}. Let $A^*$ denote the set of all (finite) words and $A^\omega$ denote the set of all (right) infinite words. A~word $w\in A^*$, $w=w_1\cdots w_n$, is \emph{accepted} by the automaton if and only if there are states $q_{i_0},\dots,q_{i_n}$ such that $q_{i_0}\in I$, $q_{i_n} \in T$ and $(q_{i_{k-1}},w_k,q_{i_k}) \in E$ for all~$k$. We call such a finite path \emph{successful}, and we call an infinite path successful if and only if infinitely many subpaths are successful. An infinite word $w\in A^\omega$ is accepted by the automaton if there exists an infinite successful path with \emph{label}~$w$. The set of all $w\in A^\omega$ that are accepted by the automaton is called its $\omega$-\emph{language}.
\end{definition}

B\"uchi automata are really helpful to describe self-similar sets. The automaton in Figure~\ref{fig:G} characterizes all infinite sequences of digits $0,1$ in base~$\alpha$ that give rise to boundary points in $\partial \K$; see~\cite{GKP98,ST03}. 

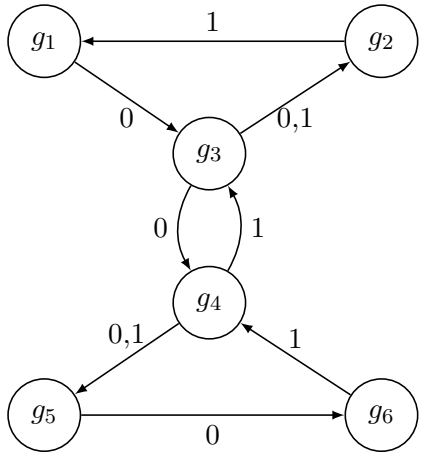
\begin{figure}[ht] 
\centerline{\begin{tikzpicture}[->,>= latex,node distance =2.8 cm,semithick]
\node [ state ] (1) {$g_1$};
\node [ state ] (2) [ right  =3.5 cm of 1] {$g_2$};
\node [ state ] (3) [ below right =0.8 cm and 1.5 cm of 1] {$g_3$};
\node [ state ] (4) [ below =1 cm of 3]  {$g_4$};
\node [ state ] (5) [ below left =0.8 cm and 1.5 cm of 4] {$g_5$};
\node [ state ] (6) [ right =3.5 cm of 5] {$g_6$};
\path (2) edge [ above ] node {1} (1) 
	  (1) edge [ below ] node {0} (3)
	  (3) edge [ below ] node {0,1} (2) 
	  (3) edge [ bend right, left ] node {0} (4)
	  (4) edge [ bend right, right ] node {1} (3) 
	  (4) edge [ above] node {0,1} (5)
	  (5) edge [ below ] node {0} (6) 
	  (6) edge [ above ] node {1} (4);
\end{tikzpicture}}
\caption{An automaton characterizing $\partial \K$ (in base $\alpha$).}\label{fig:G}
\end{figure}

Let $L_1,L_2$ two $\omega$-languages in the same alphabet accepted by $\mathcal{A}$ respectively~$\mathcal{B}$. It can be necessary to create automata accepting the union of the languages or their intersection. The union is not difficult: one just uses the union of states and edges, as well as the union of terminal and initial states. The intersection generally requires heavy computations, especially in the non-deterministic case, where a larger framework than B\"uchi automata needs to be used. But it becomes easy in some cases. We prove one particular case that will be useful to prove our main statements. 

\begin{lmm} \label{lem:intersecbuechi}
Let $L_1,L_2$ be two $\omega$-languages on the same alphabet $A$ accepted by B\"uchi automata. If one of the automata has only terminal states, then there is a B\"uchi automaton accepting $L_1\cap L_2$.
\end{lmm}

\begin{proof}
Define $\mathcal{A}\times \mathcal{B}=(Q_{\mathcal{A}}\times Q_{\mathcal{B}},A,E,I_{\mathcal{A}}\times I_{\mathcal{B}},T_{\mathcal{A}}\times T_{\mathcal{B}})$, where $E$ consists of the edges $(a,b) \overset{d} \rightarrow (a',b')$ with $a\overset{d}\rightarrow a'$ and $b\overset{d}\rightarrow b'$. Let $w\in A^\omega$ be a word that is accepted by $\mathcal{A}\times \mathcal{B}$.
Then there exists an infinite path in the automaton. 
Projecting to the first coordinate gives an infinite path through~$\mathcal{A}$. Therefore, we have $w\in L_1$ and with the same reasoning $w\in L_2$. Now let $w\in L_1\cap L_2$. There exists a path $a_0 a_1 \cdots$ through $\mathcal{A}$ and a path $b_0 b_1 \cdots$ through~$\mathcal{B}$. Then $(a_0,b_0) (a_1,b_1) \cdots$ is a path in the product automaton.
Assume w.l.o.g.\ that all states of $\mathcal{A}$ are terminal. 
Then, for every finite subpath $b_0 b_1 \cdots b_k$ accepted by~$\mathcal{B}$, the corresponding path $a_0 a_1 \cdots a_k$ in $\mathcal{A}$ is also accepted, hence $(a_0,b_0) (a_1,b_1) \cdots$ is successful.
\end{proof}

In general, if $\Delta_{p,q,r}\cap \K$ is described by a B\"uchi automaton $\mathcal{A}$ and the boundary $\partial \K$ by a B\"uchi automaton $\mathcal{G}$, then $\partial \K \cap \Delta_{p,q,r}$ is described by the product automaton $\mathcal{A} \times \mathcal{G}$. Interpreting this B\"uchi automaton as a graph directed construction for $\partial \K \cap \Delta_{p,q,r}$, we will have a way to compute the Hausdorff dimension of this set via results of Mauldin and Williams~\cite{MW88}. Let us state and prove our main statements.

\begin{thm}\label{thm:intersecbuechi}
Let $p,q,r \in \Z$, $\Delta_{p,q,r}$ as in \eqref{e:Delta} and $\K$ the Knuth Twin Dragon. Then the intersection $\K \cap \Delta_{p,q,r}$ can  be described by a B\"uchi automaton.
\end{thm}

\begin{proof}
For $s,s'\in \Z$ we define an edge relation by
\begin{equation} \label{edge}
s\overset{b}{\rightarrow} s' \Longleftrightarrow s' = p\, \mathfrak{R}(b) + q\, \mathfrak{I}(b) - 4s.
\end{equation}
Now consider a path $-r = s_0 \overset{b_1}{\rightarrow} s_1 \overset{b_2}{\rightarrow} \cdots \overset{b_n}{\rightarrow} s_n$. 
Then
\[
s_n = (-4)^n(-r) + \sum_{k=1}^n (-4)^{n-k} \big(p\, \mathfrak{R}(b_k) + q\, \mathfrak{I}(b_k)\big),
\]
i.e.,
\[
\frac{s_n}{(-4)^n} = -r +\sum_{k=1}^n \frac{p\, \mathfrak{R}(b_k) + q\, \mathfrak{I}(b_k)}{(-4)^k}.
\]
Using Lemma \ref{l:1}, we immediately get that 
\[
(x,y) = [0.b_1 b_2 b_3 \cdots]_{-4}\in \K \cap \Delta_{p,q,r} \quad \mbox{if and only if} \quad \lim_{n\to \infty} \frac{s_n}{(-4)^n} =0.
\]
We now show that the elements $s_n$ lying on paths starting with $s_0=-r$ and $\lim_{n\to \infty} \frac{s_n}{(-4)^n} =0$ are bounded by a constant $c(p,q)$. 
Indeed, we have
\[
\frac{s_n}{(-4)^n} = -r + \sum_{k=1}^n \frac{p\, \mathfrak{R}(b_k) + q\, \mathfrak{I}(b_k)}{(-4)^k} = -\sum_{k=n+1}^{\infty} \frac{p\, \mathfrak{R}(b_k) + q\, \mathfrak{I}(b_k)}{(-4)^k},
\]
and therefore
\[
|s_n| = 4^n \left |\sum_{k=n+1}^{\infty} \frac{p\, \mathfrak{R}(b_k) + q\, \mathfrak{I}(b_k)}{(-4)^k}\right| \le \frac{\max\{|p\, \mathfrak{R}(b) + q\, \mathfrak{I}(b)| : b\in \D\}}{3} = c(p,q).
\]
Defining the set of states $Q = \{s\in \Z: |s|\leq c(p,q)\} \cup \{-r\}$, $I=\{-r\}$, $T=Q$ and edges as in \ref{edge}, gives us the desired B\"uchi automaton. 
\end{proof}

\begin{thm} \label{t:Hd}
Let $p,q,r \in \Z$, $\Delta_{p,q,r}$ as in \eqref{e:Delta} and $\K$ the Knuth Twin Dragon. 
Then the Hausdorff dimension of the intersection $\partial \K \cap \Delta_{p,q,r}$ is never $\mathfrak{s}-1$, where $\mathfrak{s}$ is the Hausdorff dimension of~$\partial \K$. 
\end{thm}

\begin{proof}
The B\"uchi automaton of Theorem~\ref{thm:intersecbuechi} gives rise to a description of the intersection $\K \cap \Delta_{p,q,r}$ as the attractor of a graph directed construction (\emph{GIFS}):
\[
\K \cap \Delta_{p,q,r} = K_{-r}, \quad \mbox{with} \quad K_s=\bigcup_{s\overset{b}{\rightarrow} s'\in \mathcal{A}} \frac{K_{s'} +b}{-4} \qquad (s \in Q). 
\] 
As mentioned above, $\partial \K$ is also the attractor of a GIFS:
\[
\partial \K=\bigcup_{g\in Q'} K_g, \quad \mbox{with} \quad K_g = \bigcup_{g\overset{b}{\rightarrow} g'\in \mathcal{G}} \frac{K_{g'} +b}{-4} \qquad (g \in Q'),
\]
where $\mathcal{G}$ is the automaton characterizing $\partial \K$ in base $-4$.
The automaton $\mathcal{G}$ can be obtained from the automaton $\mathcal{G}'$ of Figure~\ref{fig:G}  as follows. 
\begin{itemize}
\item The set of states $Q'$ is the same as for $\mathcal{G}'$; all states are initial and terminal.
\item There is an edge from $g$ to $g'$ in $\mathcal{G}$ whenever there is a path of length $4$ from $g$ to $g'$ in $\mathcal{G}'$. The label of this edge in $\mathcal{G}$ is the digit vector $[d_1d_2d_3d_4]_{\alpha}$ corresponding to the labels $d_1,d_2,d_3,d_4$ in $\mathcal{G}'$ along the path of length~$4$. 
\end{itemize}
In that way, $\mathcal{A}$ and $\mathcal{G}$ are built on the same alphabet. By Lemma~\ref{lem:intersecbuechi}, the intersection $\mathcal{A}\times\mathcal{G}$ is a B\"uchi automaton describing  the intersection $\Delta_{p,q,r}\cap\partial\K$.  By Mauldin and Williams~\cite{MW88}, the Hausdorff dimension of a GIFS attractor can be computed from the spectral radius $\beta$ of the incidence matrix of a strongly connected component of the associated automaton; see further details in Remark~\ref{rem:MW}. In particular, in our case, 
\[
\mathrm{dim}_H(\partial \K \cap \Delta_{p,q,r}) = \frac{\log\beta}{\log{4}},
\]
where the involved number $\beta$ is an algebraic integer.

Now, the dimension of the boundary of the Twin Dragon is $\mathfrak{s} = \frac{\log\lambda}{\log\sqrt{2}}$, with $\lambda^3=\lambda^2+2$. To have $\frac{\log\beta}{\log{4}} = \mathfrak{s}-1$, we need $\beta =\frac{\lambda^4}{4}$. However, the minimal polynomial of $\frac{\lambda^4}{4}$ is  $4x^3-9x^2+2x-1$, thus $\frac{\lambda^4}{4}$ is not an algebraic integer.
\end{proof}

\begin{remark}\label{rem:MW}
We shortly explain why the results of Mauldin and Williams~\cite{MW88} indeed apply to our setting.  All the similarities in our graphs are contractions of the form $T(x)=\frac{x+b}{-4}$, with the same ratio $-\frac{1}{4}$. Therefore, if $G$ denotes any of our graphs, we only need to check the existence of nonoverlapping compact sets $J_1,\ldots, J_n$ (one for each node $1,\ldots,n$ of $G$) with the property
\[
\forall \,i\in\{1,\ldots,n\},\;J_i\supset\bigcup_{i\xrightarrow{T}j\in G} T(J_j),
\]
each union being nonoverlapping.

For the graph $G=\mathcal{G}$ of our paper (with states $g\in Q'$), the intersections of $\mathcal{K}$  with its six neighboring tiles in the plane tiling generated by $\mathcal{K}$ are compact sets playing the role of the $J_i$'s, that is, satisfying the above nonoverlapping conditions; see for example~\cite{AkiyamaThuswaldner05}. These intersections are exactly the sets  $K_{g}$ defined in the proof of Theorem~\ref{t:Hd}. 

Now, the graph $G=\mathcal{A}\times\mathcal{G}$ of our paper can be interpreted  as a \emph{subgraph} of~$\mathcal{G}$: taking the product of $\mathcal{A}$ and $\mathcal{G}$ means to  select paths of~$\mathcal{G}$. The states of $\mathcal{A}\times\mathcal{G}$ are of the form $(r,g)$, for some integers $r$ and $g\in Q'$.  Defining 
\[
K_{r,g} := \Delta_{p,q,-r} \cap K_{g},
\]
we obtain compact sets fulfilling the nonoverlapping requirements mentioned above.
\end{remark}

\end{section}

\section{Further results of intersections of the Twin Dragon with rational lines} \label{int}

In this section, we want to extend the work of \cite{AS05}, where the intersections with the $x$-and the $y$-axis are calculated. 
The intersections of these lines with $\partial K$ are significatively different from the expected result for intersections of fractals and lines, as they consist only of two points. 
First, we show that their result extends to uncountably many axis-parallel lines (where we do not have finite automata), and using the self-similar structure, to diagonal lines.
Then we give one example of a more complicated intersection.  

\begin{thm} \label{vert}
Let $a_1 a_2 \cdots$ be a sequence in $\{0,1\}^\omega$ not ending in $(01)^\omega$, and
\[
r = \sum_{k=1}^\infty \frac{2 a_k}{(-4)^k}.
\]
Then 
\[
\partial \K \cap \Delta_{1,0,r} = \left\{r + \left(r-\tfrac{2}{5}\right) i,\, r + \left(r+\tfrac{3}{5}\right) i\right\},
\]
and $\K \cap \Delta_{1,0,r}$ is the closed line segment $r + \left[r-\frac{2}{5}, r+\frac{3}{5}\right] i$.
\end{thm}

\begin{figure}[ht]
\centerline{\includegraphics{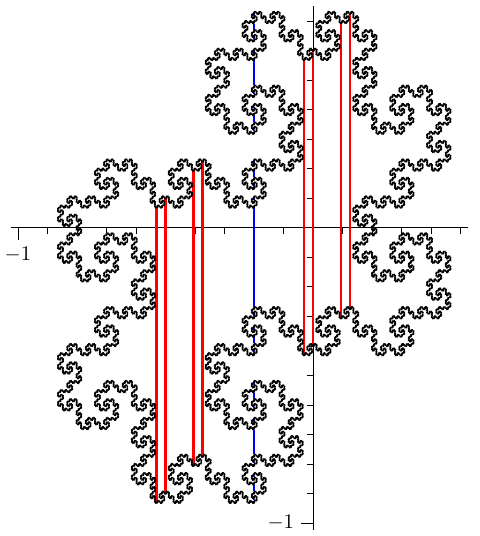}}
\caption{The Knuth Twin Dragon~$\K$ and its intersection with $\Delta_{1,0,r}$ for some $r$ as in Theorem~\ref{vert} (red) and with $\Delta_{1,0,-1/5}$ (blue).} \label{f:vert}
\end{figure}

\begin{proof}
We first use Lemma~\ref{l:1} to describe $\K \cap \Delta_{1,0,r}$, i.e., we determine the sequences $b_1 b_2 \cdots \in \D$ such that $\mathfrak{R}\left(\sum_{k=1}^\infty b_k (-4)^{-k}\right) = r$, i.e.,
\[
\sum_{k=1}^\infty \frac{2a_k-\mathfrak{R}(b_k)}{(-4)^k} = 0.
\]
Since $\mathfrak{R}(b_k) \in \{-1,0,1,2,3\}$, we have $2a_k - \mathfrak{R}(b_k) \in \{-3,-2,\dots,2,3\}$ and thus 
\[
\left|\sum_{k=n+1}^\infty \frac{2a_k-\mathfrak{R}(b_k)}{(-4)^k}\right| \le \frac{1}{4^n} \quad \mbox{for all}\ n \ge 0.
\]
Moreover, equality holds if and only if $2a_k - \mathfrak{R}(b_k)$ is alternately $3$ and~$-3$, which implies that $a_k$ is alternately $1$ and $0$, which we have excluded. 
This gives that 
\[
\left|\sum_{k=n+1}^\infty \frac{2a_k-\mathfrak{R}(b_k)}{(-4)^k}\right| < \frac{1}{4^n} \quad \mbox{and} \quad \sum_{k=n+1}^\infty \frac{2a_k-\mathfrak{R}(b_k)}{(-4)^k} = \sum_{k=1}^n \frac{\mathfrak{R}(b_k)-2a_k}{(-4)^k} \in \frac{\Z}{4^n}
\]
for all $n \ge 1$, hence $\mathfrak{R}(b_k) = 2 a_k$ for all $k \ge 1$. 
For the corresponding sequences $d_1 d_2 \cdots$ (with $\sum_{j=0}^3 d_{4k-j} \alpha^j = b_k$) this implies that
\begin{equation} \label{e:dak}
d_{4k-3} d_{4k-2} d_{4k-1} d_{4k} \in \{a_k 000, a_k 011, a_k 100, a_k111\} \quad \mbox{for all}\ k \ge 1.
\end{equation}

Now consider sequences $d_1d_2\cdots$ of the form~\eqref{e:dak} in the boundary automaton~$\mathcal{G}$ given in Figure~\ref{fig:G}.
The only paths labeled by $abcc$, $a,b,c\in\{0,1\}$, starting from $g_1$, $g_2$, $g_5$ and $g_6$ respectively are 
\[
g_1 \overset{0000}{\longrightarrow} g_6, g_1 \overset{0011}{\longrightarrow} g_2,  g_2 \overset{1000}{\longrightarrow} g_5, g_2 \overset{1011}{\longrightarrow} g_1, g_5 \overset{0100}{\longrightarrow} g_6, g_5 \overset{0111}{\longrightarrow} g_2,  g_6 \overset{1100}{\longrightarrow} g_5, g_6 \overset{1111}{\longrightarrow} g_1. 
\]
Therefore, for an infinite successful path of the form~\eqref{e:dak} starting from $g_1$, $g_2$, $g_5$ or~$g_6$, the sequence $a_1a_2\cdots$ is alternately $0$ and~$1$, which we have excluded. 
Hence, it suffices to consider paths that are in $g_3$ and~$g_4$ after $4k$ steps for all $k \ge 0$. 
From
\[
g_3 \overset{a100}{\longrightarrow} g_4 \quad \mbox{and} \quad g_4 \overset{a011}{\longrightarrow} g_3 \qquad (a \in \{0,1\}), 
\]
we see that the only points in $\partial \K \cap \Delta_{1,0,r}$ are 
\[
\begin{split}
\sum_{k=1}^\infty \frac{a_k\alpha^3}{(-4)^k} + \sum_{k=1}^\infty \frac{\alpha^6 + \alpha + 1}{16^k} & = r\,(1+i) + \frac{3i}{5}, \\
\sum_{k=1}^\infty \frac{a_k\alpha^3}{(-4)^k} + \sum_{k=1}^\infty \frac{\alpha^5 + \alpha^4 + \alpha^2}{16^k} & = r\,(1+i) - \frac{2i}{5}.
\end{split}
\]

Since $r\,(1+i) \in \K$, $\K \cap \Delta_{1,0,r}$ is the line segment between these points. 
\end{proof}

\begin{thm}
For $-\frac{8}{15} < r <  \frac{2}{15}$, we have 
\[
\begin{split}
-2i\, (\K \cap \Delta_{0,1,r/2}) & = (\K \cap \Delta_{1,0,r}) + \{0,i\}, \\
(-1+i)\, (\K \cap \Delta_{1,1,-r}) & = \K \cap \Delta_{1,0,r}\,, \\
(-1+i)\, (\K \cap \Delta_{1,-1,r/2}) & = (\K \cap \Delta_{0,1,r/2}) + \{0,1\}, \\
2\,(1+i)\, (\K \cap \Delta_{1,-1,r/2}) & =  (\K \cap \Delta_{1,0,r}) + \{-2i,-i,0,i\}.
\end{split}
\]
In particular, for $r$ as in Theorem~\ref{vert}, the sets $\K \cap \Delta_{0,1,r/2}$, $\K \cap \Delta_{1,1,-r}$ and $\K \cap \Delta_{1,-1,r/2}$ are closed line segments with endpoints
\[
\begin{split}
\partial \K \cap \Delta_{0,1,r/2} & = \partial (\K \cap \Delta_{0,1,r/2}) = \left\{-\tfrac{4}{5} - \tfrac{r}{2} + \tfrac{r}{2}\,i, \tfrac{1}{5} - \tfrac{r}{2} + \tfrac{r}{2}\, i\right\}, \\
\partial \K \cap \Delta_{1,1,-r} & = \partial (\K \cap \Delta_{1,1,-r}) = \left\{-\tfrac{1}{5} + \left(\tfrac{1}{5}-r\right) i, \tfrac{3}{10} - \left(\tfrac{3}{10} + r\right) i\right\}, \\
\partial \K \cap \Delta_{1,-1,r/2} & = \partial (\K \cap \Delta_{1,-1,r/2}) = \left\{-\tfrac{3}{5} + \tfrac{r}{2} - \tfrac{3}{5}\, i, \tfrac{2}{5} + \tfrac{r}{2} + \tfrac{2}{5}\, i\right\}.
\end{split}
\]
\end{thm}

\begin{figure}[ht]
\centerline{\includegraphics{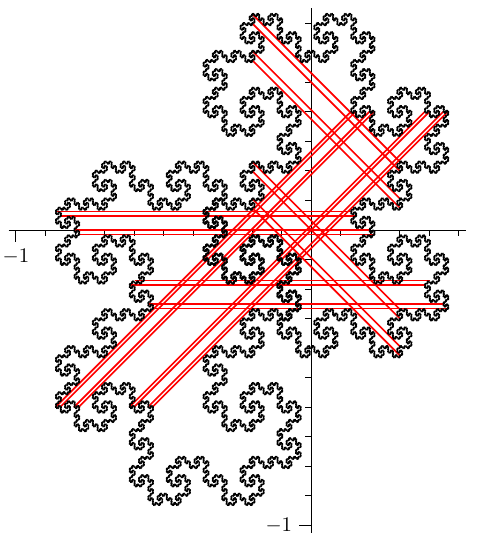}}
\caption{The intersection of $\K = \alpha^{-1} \left(\K \cup (\K+1)\right)$ with lines $\Delta_{0,1,r/2}$, $\Delta_{1,1,-r}$, and $\Delta_{1,-1,r/2}$ for some $r$ as in Theorem~\ref{vert}.} \label{f:diag}
\end{figure}

\begin{proof}
Note that $\alpha\, \K = \K \cup (\K+1)$ and 
\[
\alpha\, \Delta_{1,1,-r} = \Delta_{1,0,-r}, \quad \alpha\, \Delta_{0,1,r/2} = \Delta_{1,1,-r}, \quad \alpha\, \Delta_{1,-1,r/2} = \Delta_{0,1,r/2}.
\]
Moreover, we have 
\[
(\K+1) \cap \Delta_{1,0,r} = \emptyset = (\K-1) \cap \Delta_{1,0,r} = (\K+\alpha) \cap \Delta_{1,0,r}
\] 
since $-\frac{8}{15} < r <  \frac{2}{15}$ and
\[
\begin{split}
\min\{x \,:\, x+iy \in \K\} & = \sum_{k=1}^\infty \left( \frac{3}{(-4)^{2k-1}} + \frac{-1}{(-4)^{2k}} \right) = - \sum_{k=1}^\infty \frac{13}{16^k} = - \frac{13}{15}, \\
\max\{x \,:\, x+iy \in \K\} & = \sum_{k=1}^\infty \left( \frac{-1}{(-4)^{2k-1}} + \frac{3}{(-4)^{2k}} \right) = \sum_{k=1}^\infty \frac{7}{16^k} = \frac{7}{15}.
\end{split}
\]
Using these geometric properties, we obtain that
\[
\begin{split}
\alpha\, (\K \cap \Delta_{1,1,-r}) & = (\K \cup (\K + 1)) \cap \Delta_{1,0,r} = \K \cap \Delta_{1,0,r\,}, \\
\alpha^2\, (\K \cap \Delta_{0,1,r/2}) & = \big(\K \cup (\K+1) \cup (\K+\alpha) \cup (\K + \alpha + 1)\big) \cap \Delta_{1,0,r} \\
& = (\K \cap \Delta_{1,0,r}) \cup \big((\K + i) \cap \Delta_{1,0,r}) = (\K \cap \Delta_{1,0,r}) + \{0,i\}, \\
\alpha\, (\K \cap \Delta_{1,-1,r/2}) & = \big(\K \cup (\K + 1)\big) \cap \Delta_{0,1,r/2} =  (\K \cap \Delta_{0,1,r/2}) + \{0,1\}, \\
\alpha^3\, (\K \cap \Delta_{1,-1,r/2}) & = \alpha^2\, (\K \cap \Delta_{0,1,r/2}) - \{0,2i\} = (\K \cap \Delta_{1,0,r}) + \{-2i,-i,0,i\}.
\end{split}
\]

For $r$ as in Theorem~\ref{vert}, we have $-\frac{8}{15} < r <  \frac{2}{15}$ since 
\[
\begin{split}
\min\left\{\sum_{k=1}^\infty \frac{2 a_k}{(-4)^k} \,:\, a_1a_2\cdots \in \{0,1\}^\omega\right\} & = - \sum_{k=1}^\infty \frac{8}{16^k} = - \frac{8}{15}, \\
\max\left\{\sum_{k=1}^\infty \frac{2 a_k}{(-4)^k} \,:\, a_1a_2\cdots \in \{0,1\}^\omega\right\} & = \sum_{k=1}^\infty \frac{2}{16^k} = \frac{2}{15},
\end{split}
\]
and the minimum and maximum are attained only for the sequences $(10)^\omega$ and $(01)^\omega$, which we have excluded. 
Therefore, Theorem~\ref{vert} and the formulae above give that
\[
\begin{split}
\K \cap \Delta_{1,1,-r} & = - \tfrac{1+i}{2} \left(r\,(1+i) + \left[-\tfrac{2}{5}, \tfrac{3}{5}\right] i\right) = -r\,i + \left[-\tfrac{1}{5}, \tfrac{3}{10}\right] (1-i), \\
\K \cap \Delta_{0,1,r/2} & = \tfrac{i}{2} \left(r\,(1+i) + \left[-\tfrac{2}{5}, \tfrac{8}{5}\right]\,i\right) = r\,\tfrac{-1+i}{2} + \left[-\tfrac{4}{5}, \tfrac{1}{5}\right] , \\
\K \cap \Delta_{1,-1,r/2} & = \tfrac{1-i}{4} \left(r\, (1+i) + \left[-\tfrac{12}{5}, \tfrac{8}{5}\right] i\right) = \tfrac{r}{2} + \left[-\tfrac{3}{5}, \tfrac{2}{5}\right] (1+i),
\end{split}
\]
which proves the statements for the intersection of $\K$ with lines. 
For the intersections of $\partial \K$ with lines, it only remains to check that the points in 
\[
\alpha^{-2} \big((\K \cap \Delta_{1,0,r}) \cap ((\K \cap \Delta_{1,0,r}) + i)\big) = \left\{\tfrac{1}{\alpha^2} \left(r\,(1+i) + \tfrac{3i}{5}\right)\right\}
\]
and
\[
\alpha^{-1} \big((\K \cap \Delta_{0,1,r/2}) \cap ((\K \cap \Delta_{0,1,r/2}) + 1)\big) = \left\{\tfrac{1}{\alpha^3} \left(r\,(1+i) - \tfrac{2}{5}\,i\right)\right\}
\]
are not in~$\partial K$.  
By the proof of Theorem~\ref{vert}, the digit expansion 
\[
[.a_1100a_2011a_3100a_4011\cdots]_\alpha = r\,(1+i) + \tfrac{3}{5}\,i
\]
is given by a path starting only from~$g_3$ in the boundary automaton~$\mathcal{G}$.
Dividing by $\alpha^2$ adds $00$ in front of the expansion, but $g_3$ cannot be reached by $00$, hence $\frac{1}{\alpha^2} \left(r\,(1+i) + \frac{3i}{5}\right)$ is not on the boundary of~$K$. 
Similarly, the digit expansion 
\[
[.a_1011a_2100a_3011a_4100\cdots]_\alpha = r\,(1+i) - \tfrac{2}{5}\,i
\]
is given by a path starting from~$g_4$ in the boundary automaton~$\mathcal{G}$, and $g_4$ cannot be reached by $000$, thus $\frac{1}{\alpha^3} \left(r\,(1+i) - \frac{2}{5}\,i\right)$ is not on the boundary of~$K$. 
This proves that all intersections of $\K$ with the given lines are line segments.
\end{proof}

We can use this method to find a vertical line with a more interesting intersection. 
For example, if we look at $\Delta_{1,0,-1/4}$, we see that the only expansion $\sum_{k=1}^\infty \frac{b_k}{(-4)^k}$ with $b_k \in \mathcal{D}$ having real part $-1/4$ is $b_1b_2\cdots = 100\cdots$. 
In base $\alpha$, we must therefore have $d_1d_2d_3d_4 \in \{0001, 0101, 1010, 1110\}$, which correspond to the digits $1, 1-2i, 1+3i, 1+i \in \mathcal{D}$.
The remaining digit sequences $d_5 d_6 \cdots$ give points in $\frac{1}{\alpha^4} (\K \cap \Delta_{1,0,0})$, thus 
\[
\K \cap \Delta_{1,0,-1/4} = - \tfrac{1}{4} + \left(\left[-\tfrac{9}{10}, -\tfrac{13}{20}\right] \cup \left[-\tfrac{2}{5}, \tfrac{1}{10}\right] \cup \left[\tfrac{7}{20}, \tfrac{3}{5}\right]\right)\, i.
\]

We go on with $\Delta_{1,0,-1/4+1/16}$ and see that points in the intersection have imaginary part with an expansion in base $-4$ starting with two digits in $\{-2,0,1,3\}$ and ending with digits in $\{-1,0,1,2\}$. 
For the limit $\Delta_{1,0,-1/5}$ of lines of this form, we obtain the following intersection with~$\K$, see Figure~\ref{f:vert}.

\begin{thm}
We have
\[
\K \cap \Delta_{1,0,-1/5} = \left\{-\frac{1}{5} + \sum_{k=1}^\infty \frac{d_k}{(-4)^k}\, i \,:\, d_k \in \{-2,0,1,3\} \ \mbox{for all}\ k \ge 1\right\},
\]
and a point is in $\partial \K \cap \Delta_{1,0,-1/5}$ if and only if it is of the form $-\tfrac{1}{5} + \sum_{k=1}^\infty d_k (-4)^{-k}\, i$, where $d_1 d_2 \cdots$ is a path in the automaton in Figure~\ref{f:5}.
\end{thm}

\begin{proof}
Since $-\frac{1}{5} = \sum_{k=1}^\infty (-4)^{-k}$, we obtain in the same way as in the proof of Theorem~\ref{vert} that $\mathfrak{R}\big(\sum_{k=1}^\infty b_k (-4)^{-k}\big) = -\frac{1}{5}$ with $b_k \in \mathcal{D}$ if and only if $\mathfrak{R}(b_k) = 1$ for all $k \ge 1$, i.e., $b_k \in \{1{-}2i, 1, 1{+}i, 1{+}3i\}$. 
The corresponding 4-digit blocks in base~$\alpha$ are $0101$, $0001$, $1110$, and $1010$. 
This proves the characterization of $\K \cap \Delta_{1,0,-1/5}$.

In the boundary automaton, the digit blocks $0101$, $0001$, $1110$, and $1010$ are accepted only from $g_3$ and $g_4$, and we have the transitions 
\[
g_3 \overset{0101}{\longrightarrow} g_3,\ g_3 \overset{0001}{\longrightarrow} g_4,\ g_3 \overset{0101}{\longrightarrow} g_4,\ g_4 \overset{1010}{\longrightarrow} g_4,\ g_4 \overset{1010}{\longrightarrow} g_3,\ g_4 \overset{1110}{\longrightarrow} g_3.
\]
Taking imaginary parts of the corresponding numbers in $\mathcal{D}$ gives the automaton in Figure~\ref{f:5}.
\end{proof}

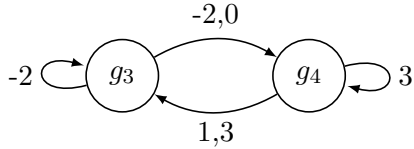
\begin{figure}[ht]
\centerline{\begin{tikzpicture}[->,>= latex,node distance=2.3 cm,semithick]
\node [ state ] (g3) {$g_3$};
\node [ state ] (g4) [ right = 1.5 cm of g3] {$g_4$};
\path (g3) edge [loop left] node{-2} ()
           (g3) edge[bend left,above] node {-2,0} (g4)
           (g4) edge[loop right] node {3} ()
           (g4) edge[bend left,below] node {1,3} (g3);
\end{tikzpicture}}
\caption{Automaton recognizing the imaginary parts of points in $\partial \K \cap \Delta_{1,0,-1/5}$ in base $-4$.} \label{f:5}
\end{figure}

\begin{thm}
The Hausdorff dimension of $\K \cap \Delta_{1,0,-1/5}$ is 1 and 
\[
\dim_H(\partial \K \cap \Delta_{1,0,-1/5}) = \tfrac{\log 3}{\log 4} \approx 0.7925 > \mathfrak{s}-1.
\]
\end{thm}

\begin{proof}
We can interpret the intersection with $\Delta_{1,0,-1/5}$ as the self-similar digit tile in $\R$ with $A=-4$ and $D=\{-2,0,1,3\}$. 
Since $D$ is a complete residue system modulo 4, this tile has non empty interior and therefore is of dimension~1. 

For the boundary, we have $\partial \K \cap \Delta_{1,0,-1/5} = K_3 \cup K_4$, with
\[
-4 K_3 = (K_3 - 2) \cup (K_4 - 2) \cup K_4, \quad -4 K_4 = (K_3 +1) \cup (K_3 + 3) \cup (K_4 + 3).
\]
Therefore, by \cite{MW88}, the Hausdorff dimension of $\partial \K \cap \Delta_{1,0,-1/5}$ is $\log \beta/\log 4$, where $\beta$ is the Perron-Frobenius eigenvalue of the matrix $\binom{1\ 2}{2\ 1}$, i.e., $\beta = 3$. 
\end{proof}
 
\subsection*{Acknowledgments.} The authors were supported  by the project I3346 of the Japan Society for the Promotion of Science (JSPS) and the FWF, the project FR 07/2019 of the Austrian Agency for International Cooperation in Education and Research (OeAD), the project PHC Amadeus 42314NC, and the project ANR-18-CE40-0007 CODYS of the Agence Nationale de la Recherche (ANR). 
\end{section}

\bibliographystyle{siam} 
\bibliography{twindragon}
\end{document}